\documentclass{amsart}
\usepackage{amssymb,latexsym, amsmath, amscd}
\usepackage{enumerate}
\usepackage{graphicx}
\usepackage{fullpage}
\usepackage[all]{xy}
\makeatletter
\@namedef{subjclassname@2010}{%
\textup{2010} Mathematics Subject Classification}
\makeatother

\newtheorem{thm}{Theorem}
\newtheorem{cor}[thm]{Corollary}
\newtheorem{lem}[thm]{Lemma}
\newtheorem{prop}[thm]{Proposition}

\theoremstyle{definition}

\numberwithin{equation}{section}

\begin{document}

\title[Cactus doodles]{Cactus doodles}

\author[J. Mostovoy and A. Rinc\'on-Prat]{Jacob Mostovoy and Andrea Rinc\'on-Prat}

\address{Departamento de Matem\'aticas, CINVESTAV\\ Col. San Pedro Zacatenco, M\'exico, D.F., C.P.\ 07360\\ Mexico}

\begin{abstract}
Cactus doodles are combinatorial/geometric objects that are related to cactus groups in the same way as knots are related to braids. We define them in terms of local moves on plane curves, show that they can be obtained from elements of the cactus group by a ``closing'' procedure and establish some of their basic properties.
\end{abstract}



\maketitle

\section{Introduction and statement of results}

\subsection{What this note is about}
There exist many groups whose elements can be represented by collections of descending curves in a plane or in space and in this respect are similar to the braid group. Examples of such generalized braid groups include virtual braid groups \cite{Ka}, groups of welded braids \cite{FRR}, twin (or planar braid) groups \cite{Kh}. In the same way as braids can be ``closed'' so as to produce knots and links, other braid-like groups give rise to knot-like objects by means of the closure operation: virtual braids produce virtual links, welded braids produce welded links and planar braids give rise to doodles (in the sense of Khovanov \cite{Kh} rather than Fenn and Taylor \cite{FT}). Many properties of these knot-like objects can be inferred from the properties of the corresponding braid-like groups although it is also true that some questions that are easy to answer for braids are hard (or still open) for links.

In the present note we introduce a new kind of knot-like objects that we call \emph{cactus doodles}. These are curves on a sphere that may have self-intersections, considered up to a number of moves that mimic the relations in the cactus group. These moves generalize the moves on usual doodles. 

We also establish some basic properties of the cactus doodles. Namely, we show that every cactus doodle arises as a closure of an element of a cactus group and prove that any two equivalent cactus doodles whose diagrams cannot be simplified can be obtained from each other by a sequence of moves that preserve the number of intersection points, and, possibly, a mirror reflection. The latter property is similar to the behaviour of the usual doodles and stands in contrast to the case of knots and links. As a corollary, we see that the usual doodles, considered up to  mirror reflections,  are a subset of the cactus doodles.

\subsection{Cactus groups}

For $n>0$, the \emph{cactus group} $J_n$ has the generators $s_{p,q}$, where $1\leq p< q\leq n$, and the following relations:
\[
\begin{array}{rcll}
s_{p,q}^2&=& 1,&\\
s_{p,q}s_{m,r}&=& s_{m,r}s_{p,q} &\quad \text{if\ } [p,q]\cap [m,r] =\emptyset,\\
s_{p,q}s_{m,r}&=& s_{p+q-r, p+q-m}s_{p,q} &\quad \text{if\ }  [m,r] \subset [p,q].
\end{array}
\]
There is a homomorphism of the cactus group $J_n$ onto the symmetric group $S_n$: it sends $s_{p,q}$ into the permutation of the ordered set $1<\ldots< n$ which reverses the order of $p, p+1,\ldots, q$ and leaves the rest of the elements unchanged. The kernel of this homomorphism is the fundamental group of the moduli space 
of \emph{stable} real rational curves with $n+1$ marked points; see, for instance \cite{EHKR} . 

Elements of $J_n$ can be interpreted as ``planar braids with self-intersections''; see \cite{MCact}. The braid corresponding to $s_{p,q}$ has one point, where $p-q+1$ strands meet, see Figure~\ref{example}. Then, the product in $J_n$ is simply the concatenation of braids. The relations have the form shown in Figure~\ref{relbr}.

\begin{figure}[ht]
$$\includegraphics[width=120pt]{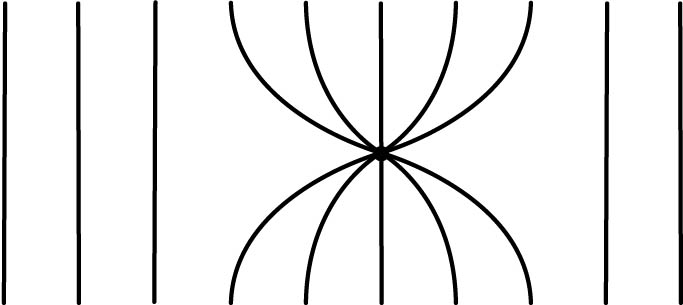}$$
\caption{A ``braid'' corresponding to $s_{4,8}$.}\label{example}
\end{figure}

\begin{figure}[ht]
$${\includegraphics[width=440pt]{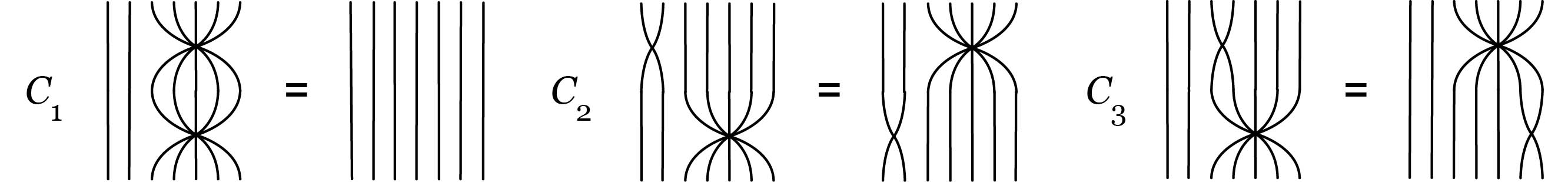}}$$
\caption{The relations in the cactus group}\label{relbr}
\end{figure}

In these terms, the homomorphism $J_n\to S_n$ can be described as follows: the permutation defined by a braid is obtained by following its strands. Braids whose self-intersections are all at different heights are in one-to-one correspondence with the words in the generators $s_{p,q}$. 

\newpage

\subsection{Doodles and cactus doodles}

Recall that a doodle is an immersed closed curve in $S^2$ whose only multiple points are transversal self-intersections. Two doodles are equivalent if they can be transformed into each other by smooth isotopies of $S^2$ and moves of the following form: 
\medskip

$$\includegraphics[width=220pt]{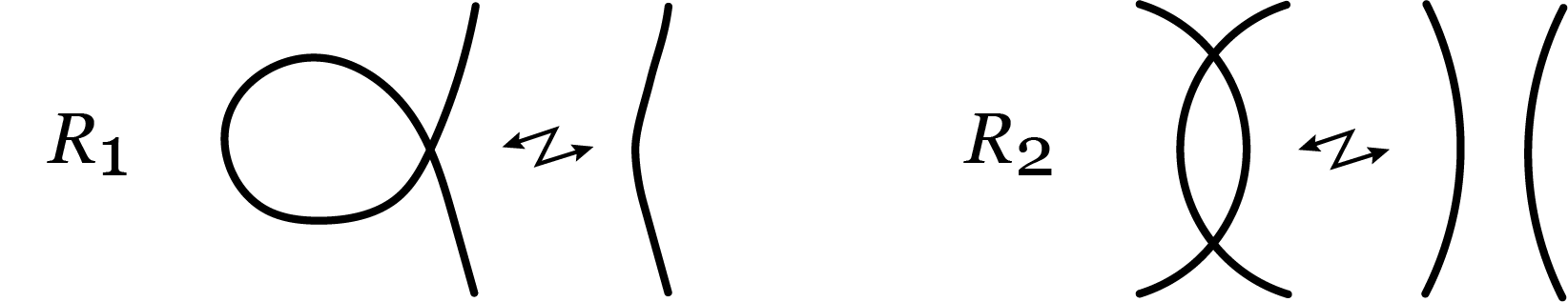}$$
\medskip

This terminology is due to Khovanov \cite{Kh}. It is a modification of the definition of Fenn and Taylor who introduced doodles in  \cite{FT}. In what follows, we consider the doodles up to isotopies and the move $R_2$ only.

A \emph{cactus doodle} is an immersed closed curve in $S^2$ whose tangent lines at each $k$-tuple point are all distinct:

$$\includegraphics[width=100pt]{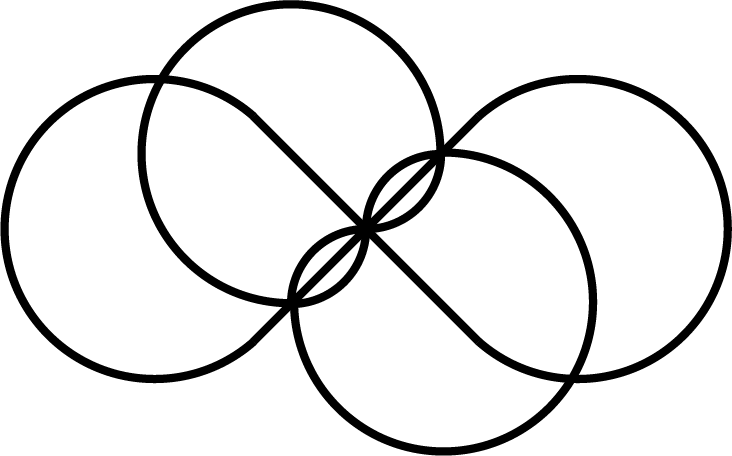}$$

We say that a cactus doodle has one component if it is parametrized by a circle (rather than several circles).

\medskip

Cactus doodles are considered modulo smooth isotopy of $S^2$ and up to \emph{elementary} moves of two types. The moves of type $\Phi$ are analogous to the relations $C_1$ in the cactus group;  the move $\Phi_n$ is defined for each $n>1$ and annihilates (or creates) a pair of $n$-tuple intersection points as in Figure~\ref{phi}.
\begin{figure}[h]
$$\includegraphics[width=440pt]{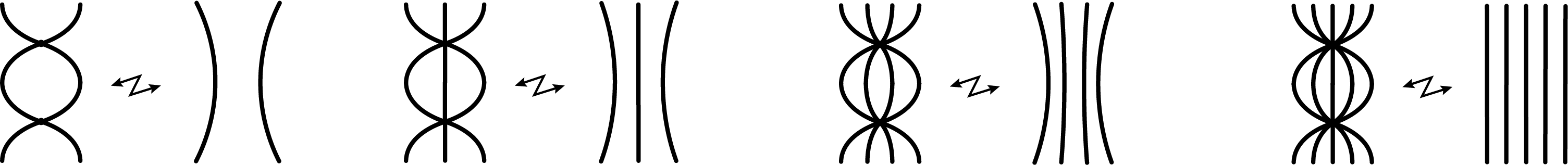}$$
\caption{The moves $\Phi_2, \Phi_3,\Phi_4$ and $\Phi_5$}\label{phi}\end{figure}

The elementary moves of type $\Psi$ are similar to the relations $C_3$ in the cactus group. For every pair $k,n$ with $1<k<n$ a move of type $\Psi_{k,n}$ passes a $k$-tuple intersection point through an $n$-tuple intersection point. There are $n-k+1$ moves of type $\Psi_{k,n}$, see Figures~\ref{psi1} and \ref{psi2}.

\begin{figure}[h]
$$\includegraphics[width=200pt]{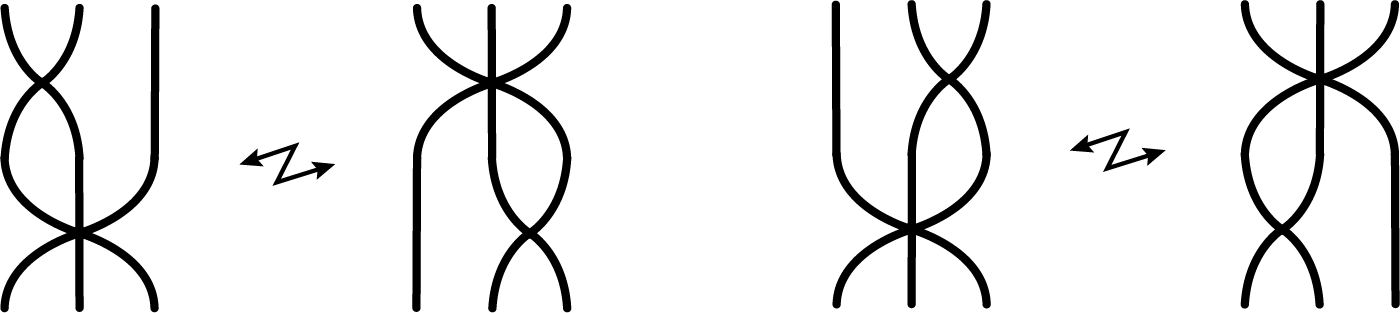}$$
\caption{The moves of type $\Psi_{2,3}$}\label{psi1}\end{figure}
\begin{figure}[h]
$$\includegraphics[width=330pt]{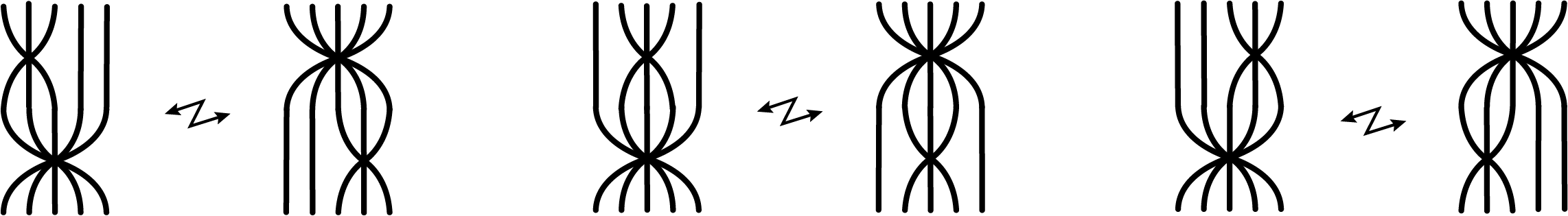}$$
\caption{The moves of type $\Psi_{3,5}$}\label{psi2}\end{figure}

 The branches of the doodle in each of the pictures may have arbitrary orientations. By saying ``cactus doodle'' we will sometimes mean the geometric curve (\emph{geometric} cactus doodle) and sometimes its equivalence class. This is similar to the usage of the term ``knot'' in knot theory and should not lead to confusion.

If, in the above definitions, we only consider immersed curves without $k$-tuple points for $k>2$, we arrive to Khovanov's definition of doodles. 

Elements of $J_n$ produce cactus doodles by joining the upper and the lower ends of each strand so that no new intersection points are created:

$$\includegraphics[width=150pt]{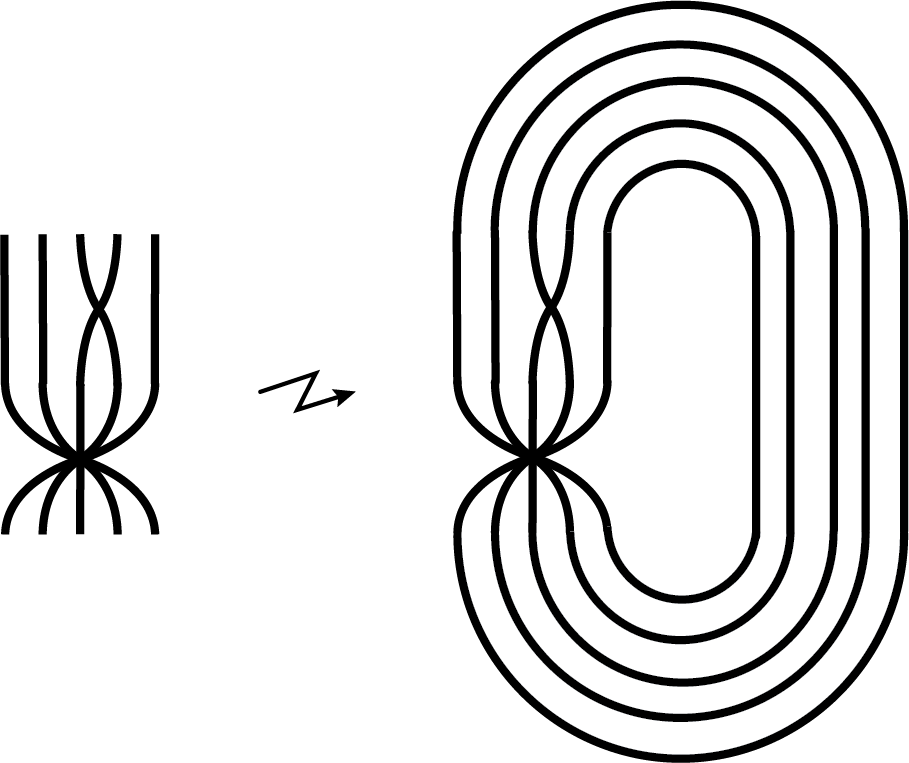}$$
\medskip

\noindent Here, we assume that all the strands of the planar braids are oriented, say, downwards.
\begin{thm}\label{main} Each cactus doodle is equivalent to the closure of some element of $J_n$.
\end{thm}

A cactus doodle $\mathcal{C}'$ is a \emph{mirror image} of a cactus doodle $\mathcal{C}$ if it is obtained from $\mathcal{C}$ by a reflection with respect to some line.  The above theorem has the following consequence whose proof is illustrated by Figure~\ref{flip}:

\begin{cor} Each cactus doodle is equivalent to its mirror image.
\end{cor}

\begin{figure}[ht]
\includegraphics[width=350pt]{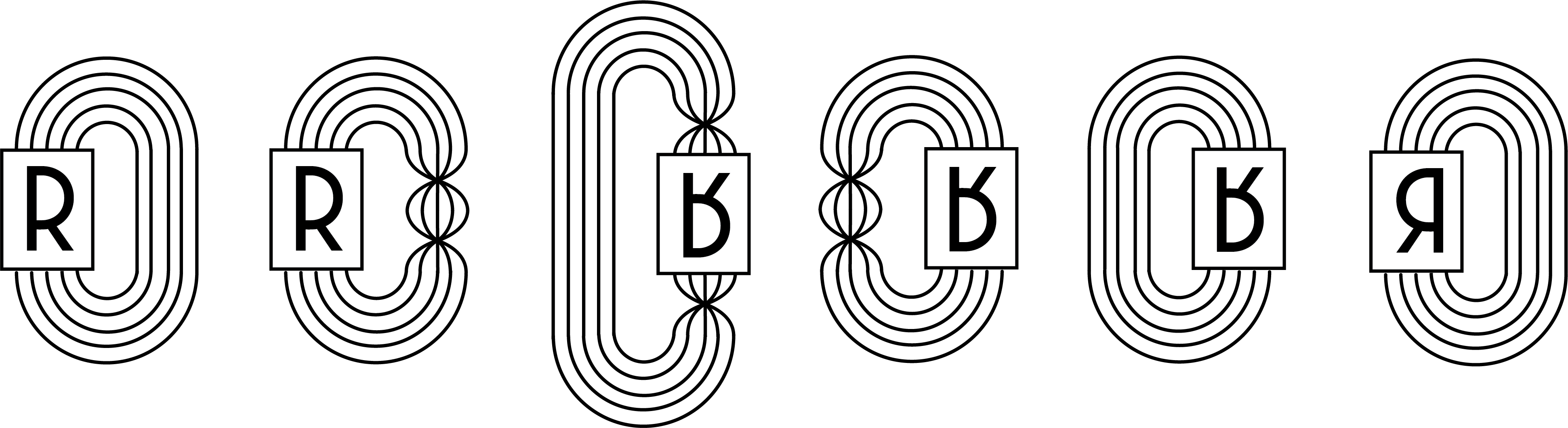}
\caption{The equivalence between a closure of an element of $J_n$ and its mirror image}\label{flip}
\end{figure}

We should point out that a usual doodle is not always equivalent to its mirror image. In particular, there exist non-equivalent doodles which are equivalent as cactus doodles. By Theorem~\ref{min} below, if such doodles consist of one component, they are necessarily mirror images of each other.

\medskip

The problem of distinguishing doodles, unlike the problem of distinguishing knots, is easy. Call a doodle \emph{minimal} if none of its intersection points can be removed by an elementary move. It is easy to show that any doodle is equivalent to precisely one minimal doodle; in order to construct this minimal doodle, one simply  removes all the intersection points that can be removed.

The situation with cactus doodles turns out to be broadly similar, although the minimal doodle is not necessarily unique.  
Call a cactus doodle \emph{minimal} if any sequence of moves that reduces its number of intersection points must first increase its number of intersection points.

\begin{thm}\label{min} If two equivalent cactus doodles with one component are minimal, they can be transformed into each other by elementary moves of type $\Psi$ and mirror reflections. In particular, all equivalent minimal doodles have the same number of intersection points.
\end{thm}
 


There are many natural questions that can be asked about cactus doodles. The first of these is the precise description of the closure map; in other words, the Markov-type theorem for cactus doodles (for usual doodles it has been established by Gotin \cite{Got}). Another question is whether there exists a ``fundamental group'' for cactus doodles similar to the one constructed for doodles by Khovanov \cite{Kh}. One may look for other invariants such as the Polyak-Viro-type invariants which classify doodles as shown by Merkov \cite{Merk}, or for polynomial-type invariants such as that of \cite{Juy}. Finally, one may wonder whether cactus doodles arise in some other context. We do not attempt to answer any of these questions here. 

\section{Proof of Theorem~\ref{main}}

Call an intersection point of a geometric cactus doodle \emph{coherent} if there exists a vector whose scalar product  with any tangent vector to the cactus doodle at this point is positive. In the figure below, the intersection point on the left is coherent while the one on the right is not:

$$\includegraphics[width=80pt]{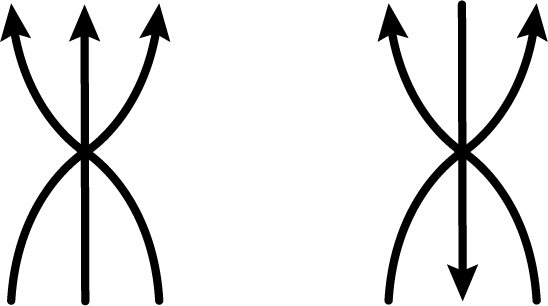}$$
\medskip

Any intersection point of two branches of a cactus doodle is coherent. Also, observe that a geometric cactus doodle obtained by closing a planar braid only has coherent intersection points.

\begin{lem}\label{coherent} All intersection points of a cactus doodle can be made coherent by an application of several elementary moves.
\end{lem}
\begin{proof}
Indeed, the following sequence of moves:

$$\includegraphics[width=150pt]{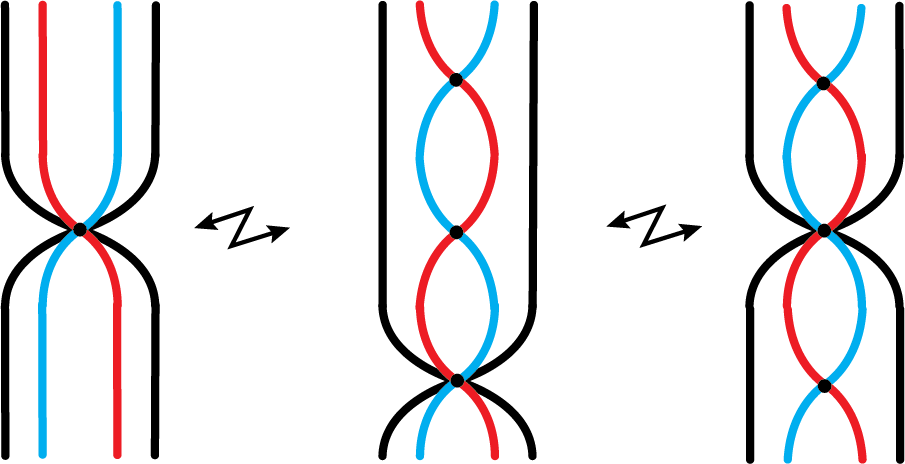}$$

\noindent interchanges two adjacent branches of a cactus doodle at an intersection point and adds two double intersection points. Interchanging different adjacent branches of the cactus doodle repeatedly, we can obtain a picture where the original intersection point becomes coherent and several double intersections are added. These latter are always coherent.
\end{proof}

Think of $S^2$ as $\mathbb{R}^2\cup\{\infty\}$ and write $0$ for the origin in $\mathbb{R}^2$.  
A \emph{special point} on a cactus doodle $\mathcal{C}$ is either an intersection point of $\mathcal{C}$ or a tangency point of $\mathcal{C}$ with some ray emanating from 0.

Given a cactus doodle $\mathcal{C}$, without loss of generality we can assume that:
\begin{enumerate}
\item each intersection point of  $\mathcal{C}$ is coherent;
\item neither $0$ nor $\infty$ lie on $\mathcal{C}$;
\item $\mathcal{C}$ only has a finite number of special points;
\item there is at most one special point on any ray emanating from 0;
\item any tangent vector at any intersection point of $\mathcal{C}$ is directed counterclockwise with respect to 0.
\end{enumerate}
Indeed, (1) follows from Lemma~\ref{coherent}, (2)--(4) can be achieved for any $\mathcal{C}$ by an arbitrarily small isotopy of $S^2$. As for (5), an isotopy that achieves it, rotates, if necessary, each intersection point so that the tangent vectors point in the desired direction; it only moves points in a small neighbourhood of the intersection; see Figure~\ref{turn}.
\begin{figure}[h]
$$\includegraphics[width=200pt]{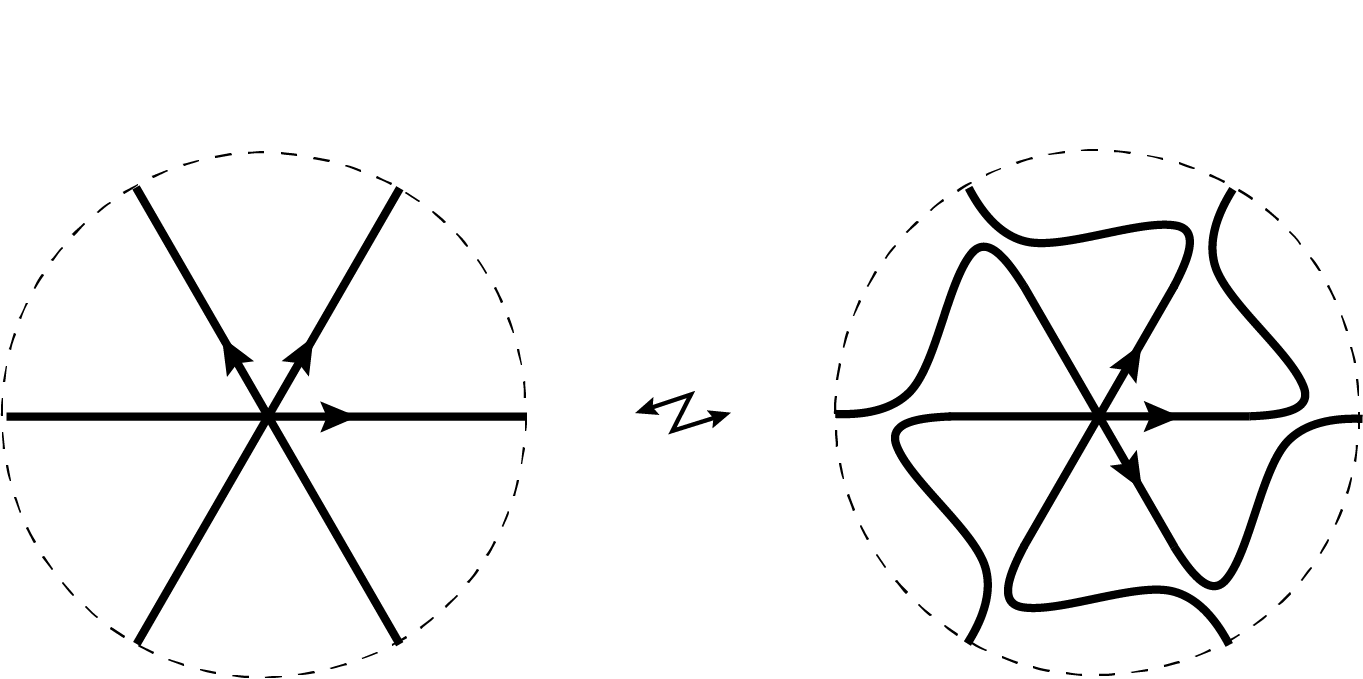}$$
\caption{Rotating an intersection point}\label{turn}\end{figure}
The rest of the proof follows Khovanov's argument in \cite{Kh}. 

If $\mathcal{C}$ has no point where the tangent vector points in the clockwise direction with respect to 0, it is a closure of a ``planar braid''. The ``braid'' depends on the choice of a ray emanating from 0 without special points of  $\mathcal{C}$ on it, see Figure~\ref{braid}.
\begin{figure}[h]
$$\includegraphics[width=200pt]{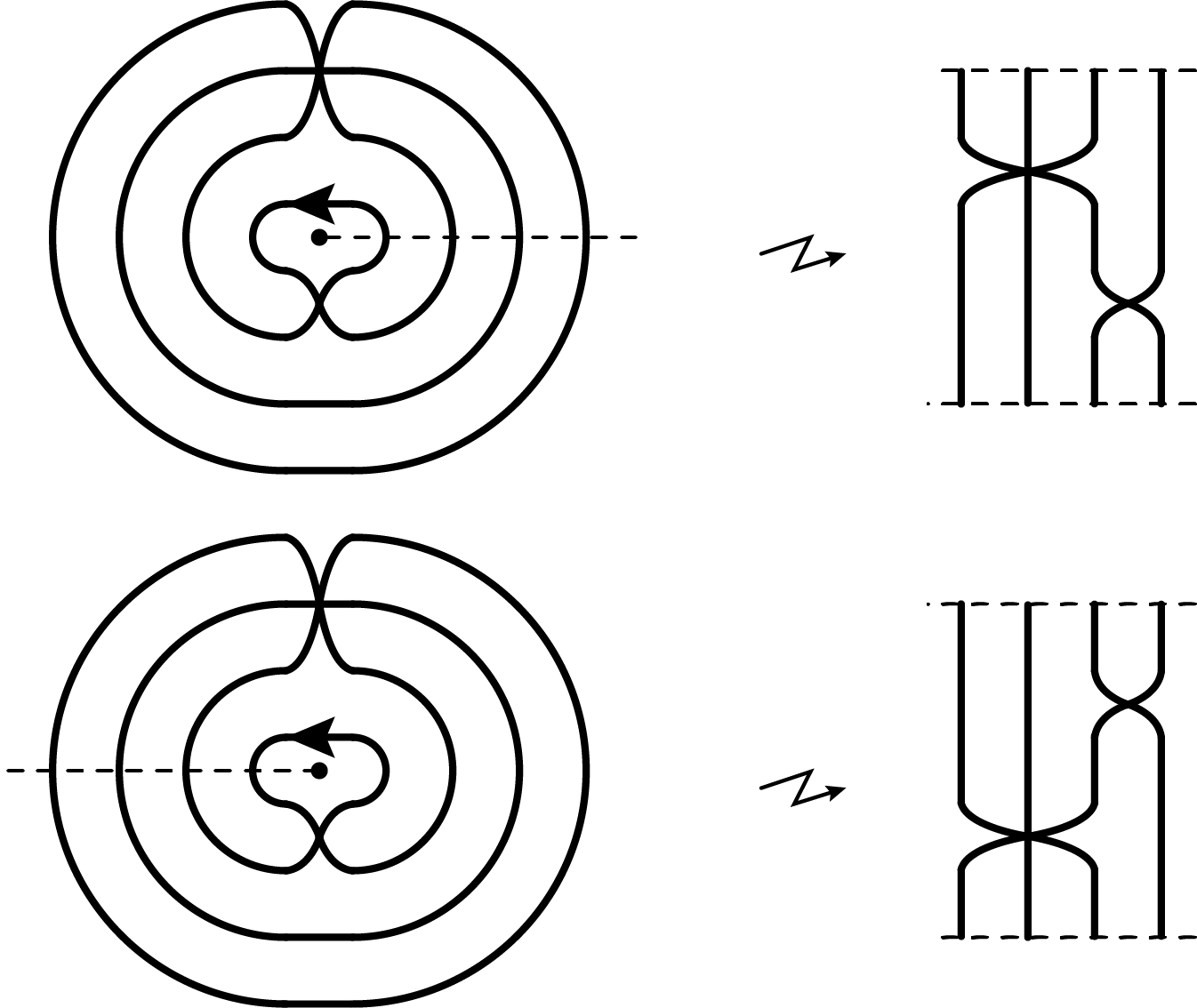}$$
\caption{Obtaining an element of the cactus group from a cactus doodle}\label{braid}\end{figure}
In general, the set of points of $\mathcal{C}$ where the tangent vector points in the clockwise direction with respect to 0 is the union of a finite number of open intervals $I_1,\ldots, I_n$ whose endpoints are (some of the) special points of $\mathcal{C}$. Note that the interiors of these intervals do not contain special points of $\mathcal{C}$. 
We will show how to modify $\mathcal{C}$ on $I_n$ by isotopies and elementary moves in such a way that the resulting cactus doodle has the tangent vector in the clockwise direction on $I_1,\ldots, I_{n-1}$ only. 

Let $A$ be  a special point of $\mathcal{C}$ such that the ray $0A$ intersects $I_n$. We say that $A$ is \emph{inner} if it lies between $0$ and the intersection of $0A$ with $I_n$ and \emph{outer} otherwise. In addition, if  $A$ is the initial  (final) endpoint of $I_n$, we call it inner if the tangent vector to  $\mathcal{C}$ at $A$ points towards 0  (towards infiitity, respectively). An endpoint of $I_n$ is outer if it is not inner.

Subdivide the interval $I_n$ by points $B_0, B_1, \ldots, B_r$ into open intervals $B_iB_{i+1}$ in such a way that 
\begin{enumerate}
\item $B_0$ is the initial and $B_r$ the final endpoint of $I_n$ and all the $B_iB_{i+1}$ are pairwise disjoint;
\item none of the points $B_i$ with $0<i<r$ lies on a ray $0A$ where $A$ is special;
\item the closure $\overline{B_iB}_{i+1}$ of ${B_iB_{i+1}}$ intersects at least one ray $0A$ where $A$ is special, for each $0\leq i< r$;
\item if $\overline{B_iB}_{i+1}$ intersects several rays $0A$ with $A$ special, such $A$ are either all inner or all outer;  
\item if $\overline{B_iB}_{i+1}$ with $0\leq i< r-1$ intersects a rays $0A$ with $A$ inner, $\overline{B_{i+1}B}_{i+2}$ intersects $0A$ with $A$ outer and vice versa.
\end{enumerate}
These definitions are illustrated in Figure~\ref{Kh1}. 
\begin{figure}[h]
$$\includegraphics[width=200pt]{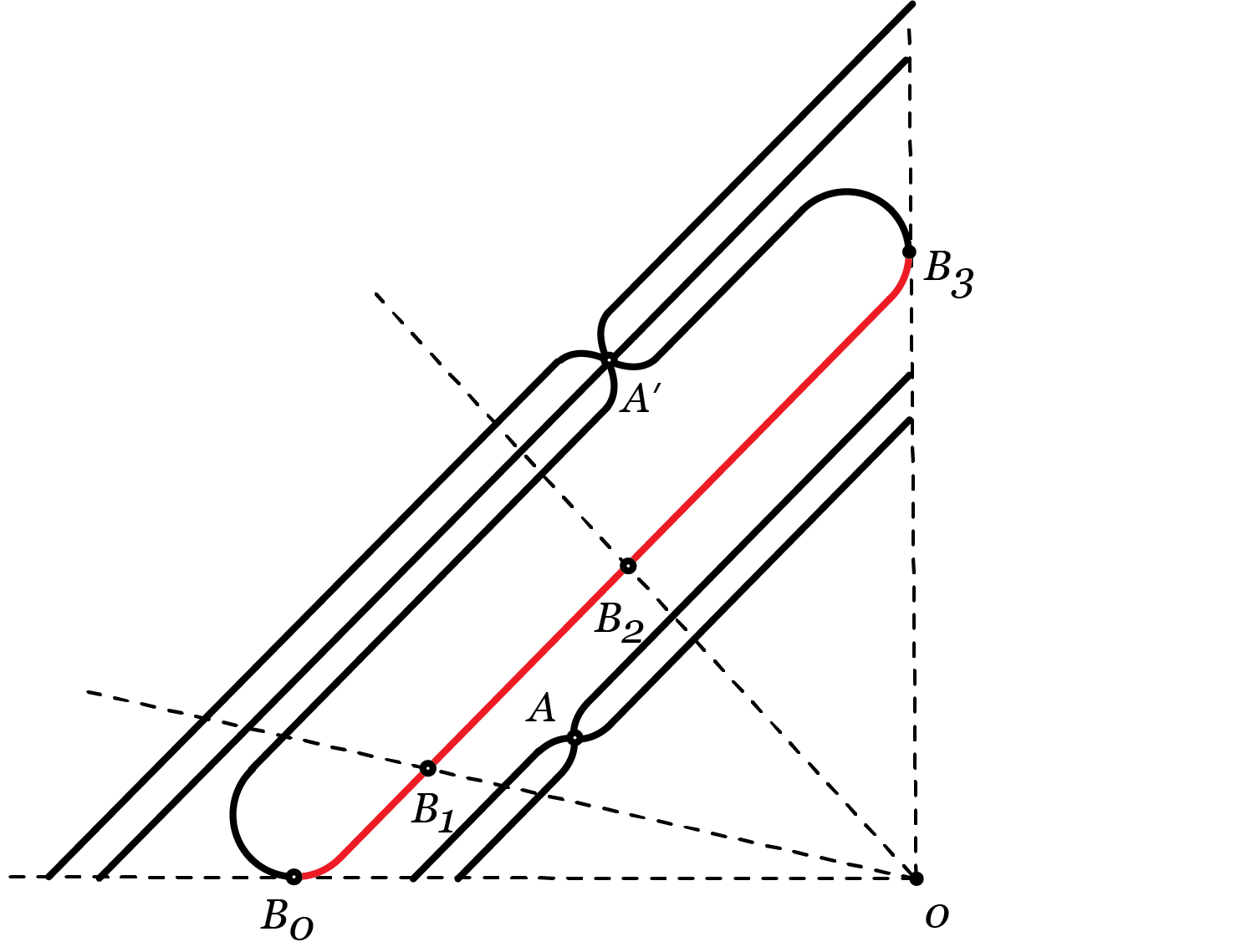}$$
\caption{The figure shows the part of a cactus doodle contained between the rays $0B_0$ and $0B_3$. The interval $I_n$ is shown in red. The special points $B_0$, $A'$ and $B_3$ are outer, $A$ is inner.}\label{Kh1}\end{figure}

Now, if $B_iB_{i+1}$ intersects $0A$ with $A$ inner, this interval can be pushed towards infinity by means of isotopies and elementary moves $\Phi_2$. Likewise,  if $B_iB_{i+1}$ intersects $0A$ with $A$ outer, it can be pushed towards 0 as in Figure~\ref{Kh2}. By pushing it over infinity or over 0 one then achieves the change of direction of the whole $I_n$: instead of turning around 0 in the clockwise direction (say, by an angle $\alpha$), the modified interval turns around 0 in the counterclockwise direction by the angle $2\pi r-\alpha$.

\begin{figure}[h]
$$\includegraphics[width=300pt]{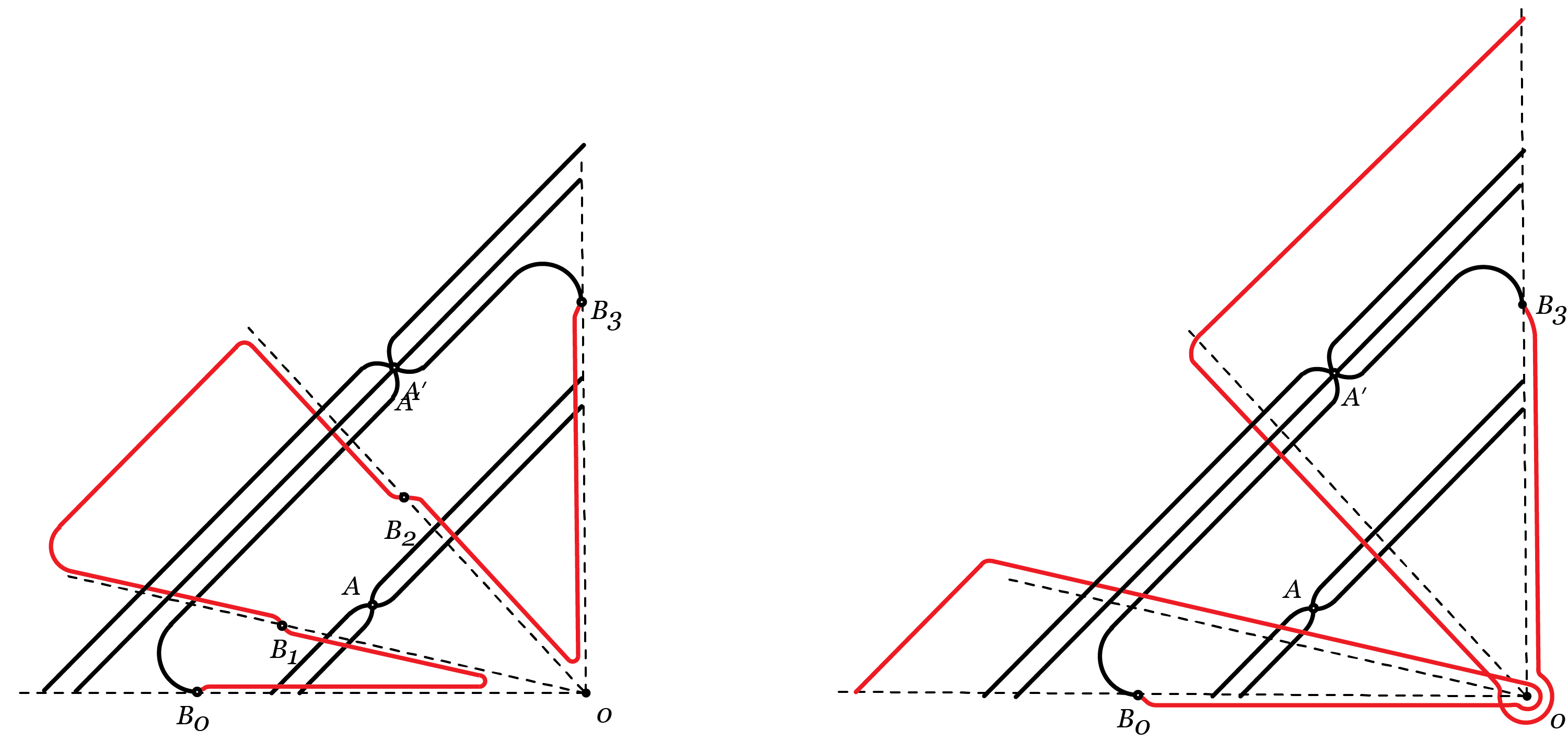}$$
\caption{Changing the direction of the interval $I_n$}\label{Kh2}\end{figure}

\section{Gauss diagrams of cactus doodles}

A usual geometric doodle $S^1\to S^2$ can be reconstructed, up to isotopy, from its \emph{Gauss diagram}. This diagram consists of $S^1$ on which each pair of the preimages of each intersection point is connected with a directed chord. At each intersection point, the branches of the doodle are ordered in such a way that their tangent vectors form a positive basis; the arrow on the chord goes from the first branch to the second branch. Gauss diagrams of knot projections are widely used in knot theory; note, however, that the direction of an arrow in the Gauss diagram of a knot projection has a different meaning. According to Arnol'd \cite{Ar}, abstract Gauss diagrams were first studied by Gauss who was interested in the question of their realizability by immersed plane curves. 
\medskip

The definition of a Gauss diagram can be extended to cactus doodles. 
Given a geometric cactus doodle $C$ parametrized by a union of circles $M=S^1\sqcup\ldots\sqcup S^1$, a \emph{rough diagram of $C$} is a copy of $M$ (called the \emph{skeleton} of the diagram) together with a finite set of marked points on it. It is constructed as follows:
\begin{enumerate}
\item label the intersection points of $C$ by different labels;
\item mark on the skeleton all the preimages of the intersection points;
\item label each marked point by the label of the corresponding intersection point.
\end{enumerate}
We call the set of all preimages of an intersection point with label $L$ the \emph{singular set with label $L$}. We do not distinguish rough diagrams obtained from each other by orientation-preserving diffeomorphisms of the skeleton which send each connected component to itself.

Here is an example of a cactus doodle and a rough diagram of it:
$$\includegraphics[width=300pt]{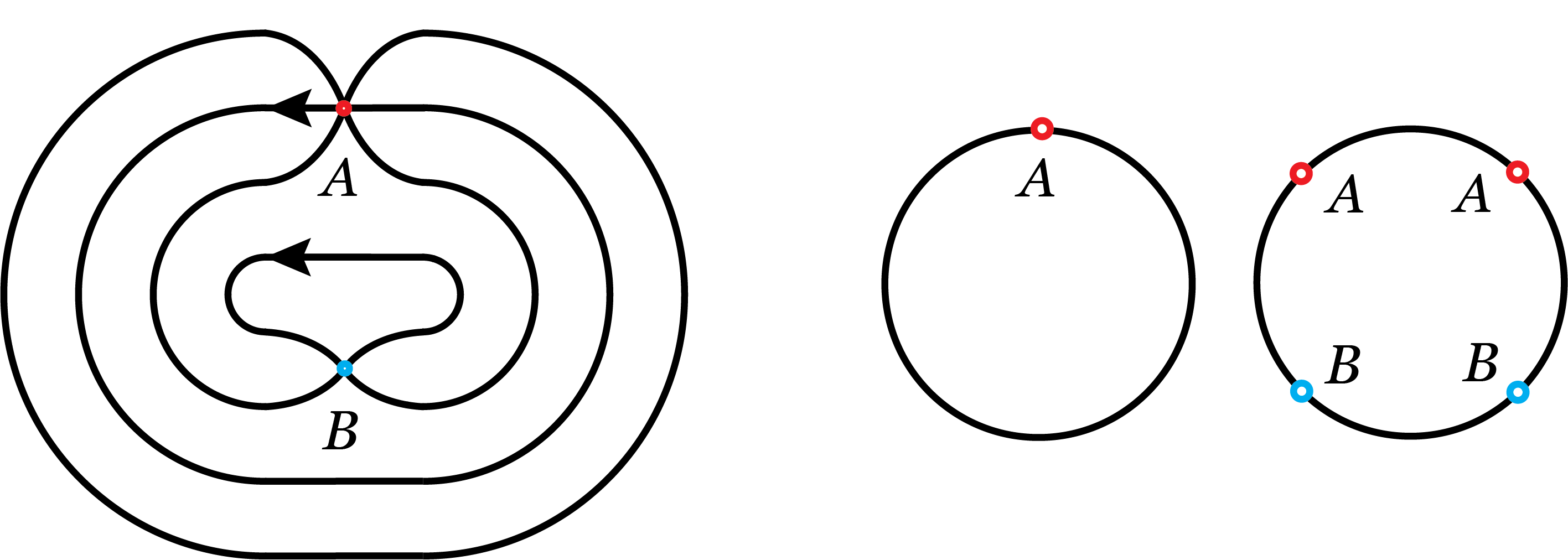}$$
The components of the skeleton are assumed to be oriented counterclockwise (although in this particular example it is irrelevant).
 
The elementary moves on cactus doodles induce moves on their rough diagrams; these are shown in Figure~\ref{rdmoves}.
\begin{figure}[h]
$$\includegraphics[width=350pt]{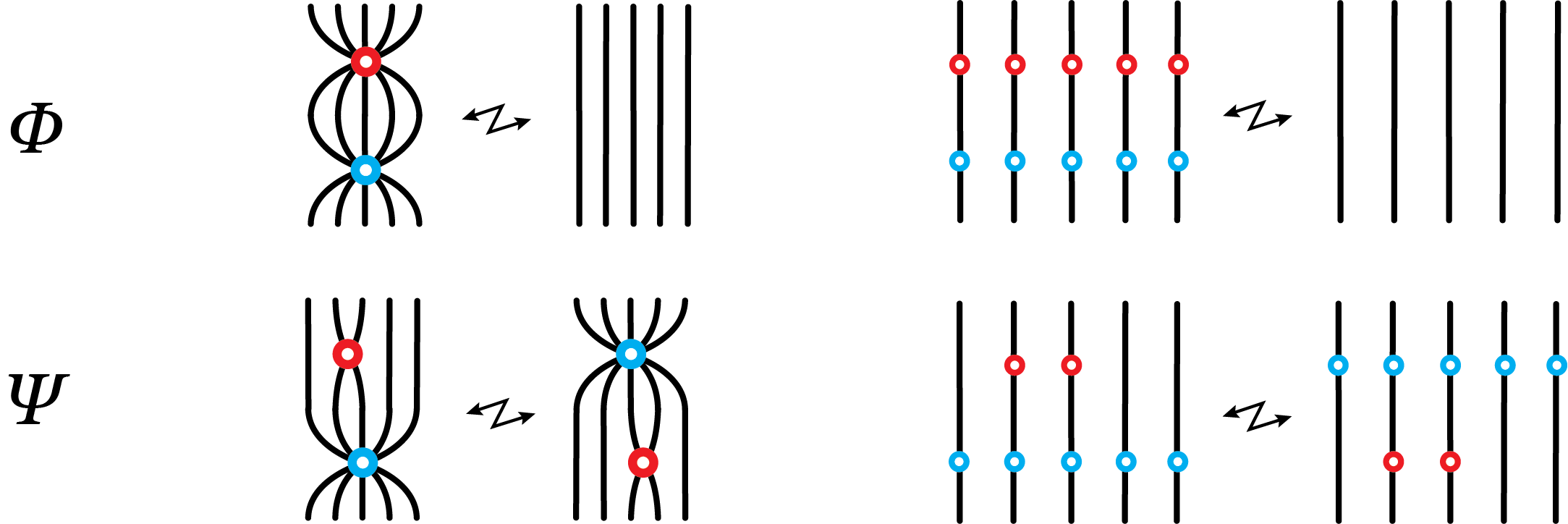}$$
\caption{Elementary moves on rough diagrams}\label{rdmoves}
\end{figure}

We will refer to all marked configuration of points on a union of circles as rough diagrams, whether or not they are \emph{realizable}, that is, whether or not they come from a cactus doodle. The above moves make sense for all, not necessarily realizable, rough diagrams.

A rough diagram of a cactus doodle does not carry sufficient information for reconstructing the cactus doodle up to isotopy. The following cactus doodles are different but their rough diagrams are the same:
\medskip

$$\includegraphics[width=250pt]{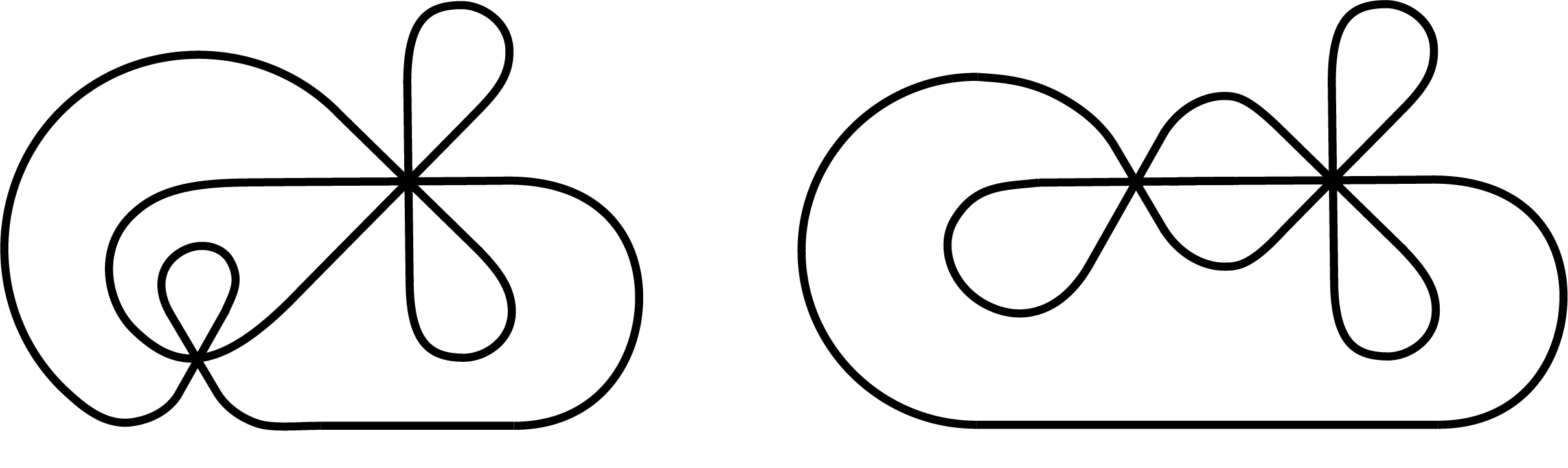}$$

\medskip

In order to be able to reconstruct a cactus doodle from a rough diagram, one should add the information about the orientations and the cyclic order of the branches to each singular set. Define an \emph{oriented cyclic order} on a set $X$ as a cyclic order on $X$ together with its lifting to a cyclic order on the set $X\times\{\pm 1\}$. In other words, it is a cyclic order on $X\times\{\pm 1\}$ such that if $(x_2, \varepsilon_2)$ follows $(x_1, \varepsilon_1)$, then  $(x_2, -\varepsilon_2)$ follows $(x_1, -\varepsilon_1)$ for $x_1,x_2\in X$ and $\varepsilon_1,\varepsilon_2\in\{\pm 1\}$. If $A\subset X$ is any subset, an oriented cyclic order on $X$ induces an oriented cyclic order on $A$. The \emph{reversal} of an oriented cyclic order on $A\subseteq X$ is the oriented cyclic order on $X$ obtained by inverting the order on each maximal subset of consecutive elements of $A\times\{\pm 1\}$.

The set $X$ of branches of a cactus doodle at any given intersection point has a natural oriented cyclic order on it. Let $(x,-1)$ denote the initial and $(x,1)$ the final point of a branch $x\in X$. Then, listing the endpoints of the  branches in the counterclockwise order, one obtains an oriented cyclic order on $X$:

\begin{figure}[h]
$$\includegraphics[width=250pt]{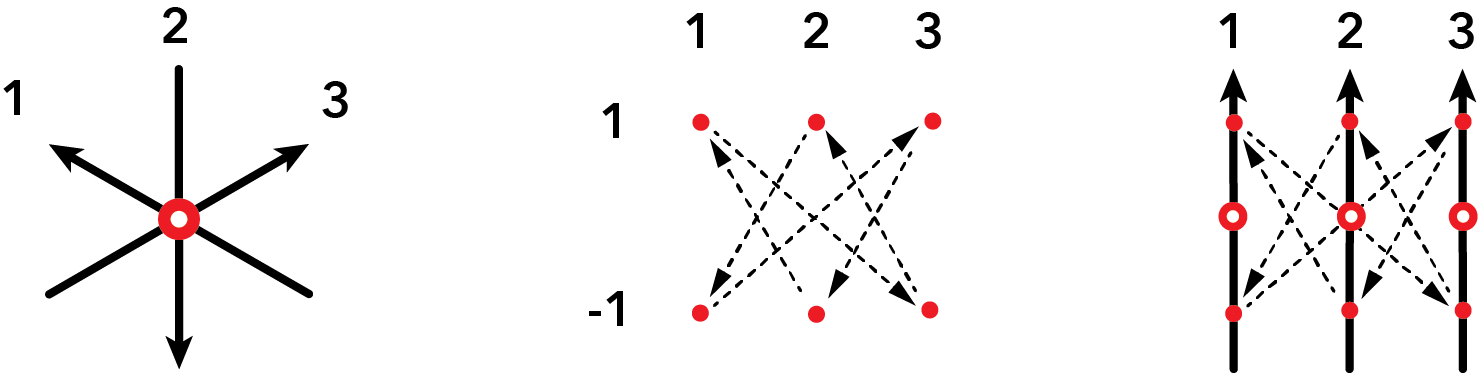}$$
\caption{An intersection point of a cactus doodle (left), the oriented cyclic order on the set of branches at the intersection point (middle), the corresponding singular set on its Gauss diagram (right).}
\end{figure}

We define a \emph{Gauss diagram}  as a rough diagram together with an oriented cyclic order on each singular set. For a 
geometric cactus doodle ${C}$, the Gauss diagram of ${C}$ is the rough diagram of ${C}$ together with the oriented cyclic order of the branches of ${C}$ on each singular set. The elementary moves for the Gauss diagrams are defined as in Figure~\ref{rdmoves} with the following additional conditions: in each diagram, the oriented cyclic order on the singular set with the red label 
is the opposite of the oriented cyclic order induced by that on the singular set with the blue label;  the move $\Psi$ reverses the oriented cyclic order on the red singular set and on the corresponding points of the blue singular set:

$$\includegraphics[width=300pt]{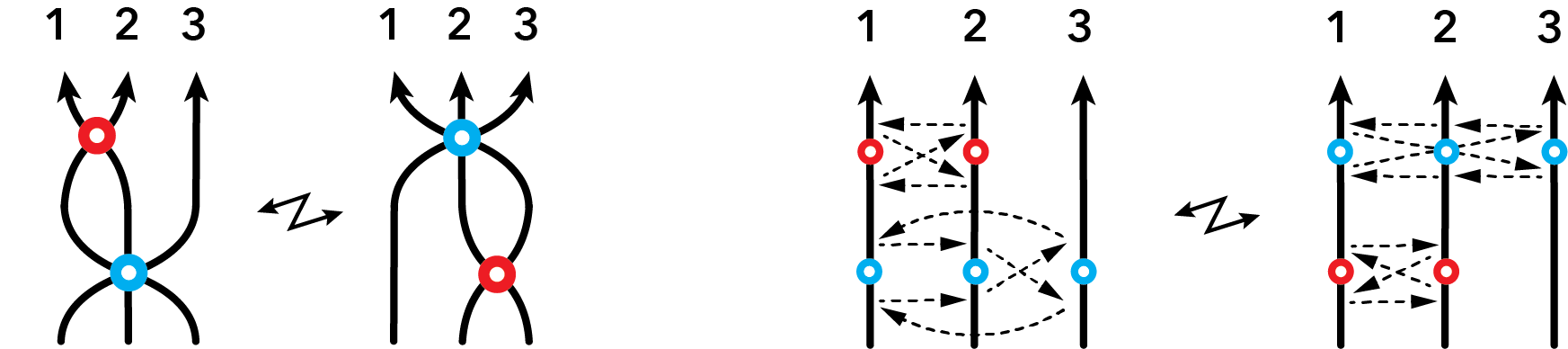}$$

\medskip

A Gauss diagram does not necessarily come from a cactus doodle, but when it does, it often characterizes the cactus doodle completely.

Let us call a geometric cactus doodle \emph{connected} if its image is connected as a subset of $S^2$.
\begin{thm}\label{GD} 
Two connected geometric cactus doodles with the same Gauss diagram are isotopic in $S^2$.
\end{thm}
\begin{proof}
Indeed, the complement to a connected cactus doodle in $S^2$ consists of a finite number of regions homeomorphic to discs. The boundary of each region consists of a finite number of intervals whose endpoints are the intersection points. The Gauss diagram contains the complete information on the order and orientation of the intervals in the boundary of each disc and on the way the discs are glued with each other along the doodle. 
\end{proof}

The following (obvious) statement holds for all, not necessarily connected, cactus doodles:

\begin{lem}\label{dia} Let $C$ be a geometric cactus doodle with the Gauss diagram $D$. If the Gauss diagram $D'$ is realizable and is obtained from $D$ by an elementary move, then there exists a unique cactus doodle $C'$  obtained from $C$ by an elementary move and whose Gauss diagram is  $D'$.
\end{lem}
In other words, elementary moves between realizable Gauss diagrams are themselves uniquely realizable. Another useful statement whose proof is obvious is
\begin{lem}\label{rea} If the Gauss diagram $D'$ is obtained from a realizable  diagram $D$ by an elementary move of type $\Psi$, then $D'$ is itself realizable. The same is true if $D'$ is obtained from  $D$ by an elementary move of type $\Phi$ that eliminates a pair of singular sets.
\end{lem}

Note that an elementary move of type $\Phi$ that creates a pair of singular sets in a realizable diagram may result in a non-realizable diagram.

Finally, let us note that that the operation of mirror reflection 
of a cactus diagram can be defined on the level of Gauss diagrams. The mirror image $D^*$ of a Gauss diagram $D$ is obtained from $D$ by taking a mirror image of the rough diagram of $D$ and endowing each singular set with a cyclic order opposite to that in $D$.

\section{Proof of Theorem~\ref{min}}
Assume that there exist two equivalent minimal geometric cactus doodles which cannot be transformed into each other by elementary moves or mirror reflections without increasing the number of intersection points. Then, one can find at least one sequence  of geometric cactus doodles $C_0,\ldots, C_r$ with $r>1$ such that:
\begin{enumerate}
\item $C_i$ is obtained from $C_{i-1}$ by an elementary move of type $\Psi$ or a mirror reflection for each $i$ with $1<i<r$;
\item $C_1$ is obtained from $C_0$ by inserting a pair of intersection points by means of a move of type $\Phi$;
\item $C_r$ is obtained from $C_{r-1}$ by eliminating a pair of intersection points by means of a move of type $\Phi$;
\item any sequence of elementary moves and mirror reflections connecting $C_0$ to $C_r$ must include a geometric cactus doodle which has more intersection points than $C_0$ and $C_r$.
\end{enumerate} 
Indeed, if no sequence of geometric cactus doodles satisfying (1-3) satisfies (4), any pair of equivalent geometric cactus doodles with the same number of intersection points can be transformed into each other without increasing the number of these points. 

Theorem~\ref{min} will be proved as soon as we establish
\begin{prop}\label{xz}
The conditions (1-3) are in contradiction with the condition (4) above.
\end{prop}

In view of Lemma~\ref{dia}, Proposition~\ref{xz} is a corollary of the corresponding statement for Gauss diagrams:
\begin{prop}\label{xz2}
Let $D_0,\ldots, D_r$ be a sequence of realizable Gauss diagrams such that:
\begin{enumerate}
\item $D_i$ is obtained from $D_{i-1}$ by an elementary move of type $\Psi$ for each $i$ with $1<i<r$;
\item $D_1$ is obtained from $D_0$ by inserting a pair of singular sets by means of the move of type $\Phi$;
\item $D_r$ is obtained from $D_{r-1}$ by eliminating a pair of singular sets by means of the move of type $\Phi$;
\end{enumerate} 
Then, there exists a sequence of elementary moves connecting $D_0$ either to $D_r$ or its mirror image $D_r^*$ in which every diagram is realizable and has the same or smaller number of singular sets as $D_0$ and $D_r$.
\end{prop}

\begin{proof} 

Let  $A$ and $A'$ be the labels of the singular sets that are created by the move that transforms $D_0$ to $D_1$ and denote by  $B$ and $B'$ the  labels of the singular sets that are annihilated by the move that transforms $D_{r-1}$ into $D_r$.

Let us first consider the case $\{A,A'\}=\{B,B'\}$ and set $|A|=|A'|=n$. Take $n$ disjoint segments $I_1,\ldots I_n$ on the skeleton of $D_1$ which connect the points of $A$ with the points of $A'$ and contain no point of any other singular set. We can think that the moves that transform $D_{i}$ into $D_{i+1}$ for each $0< i<{n-1}$  fix the segments $I_j$,  while possibly moving points of other singular sets to these segments. The complement of the $I_j$ in the skeleton of each $D_i$ also consists of $n$ disjoint segments; the move that transforms $D_{r-1}$ into $D_r$ and annihilates $A$ and $A'$ can have two forms: either it shrinks the $I_j$ or it shrinks the complement of the $I_j$. 

Let $Q_1,\ldots Q_{m_i}$ be the singular sets of $D_i$ which have some of their points on the segments $I_j$; see Figure~\ref{seq}.
\begin{figure}[h]
$$\includegraphics[width=300pt]{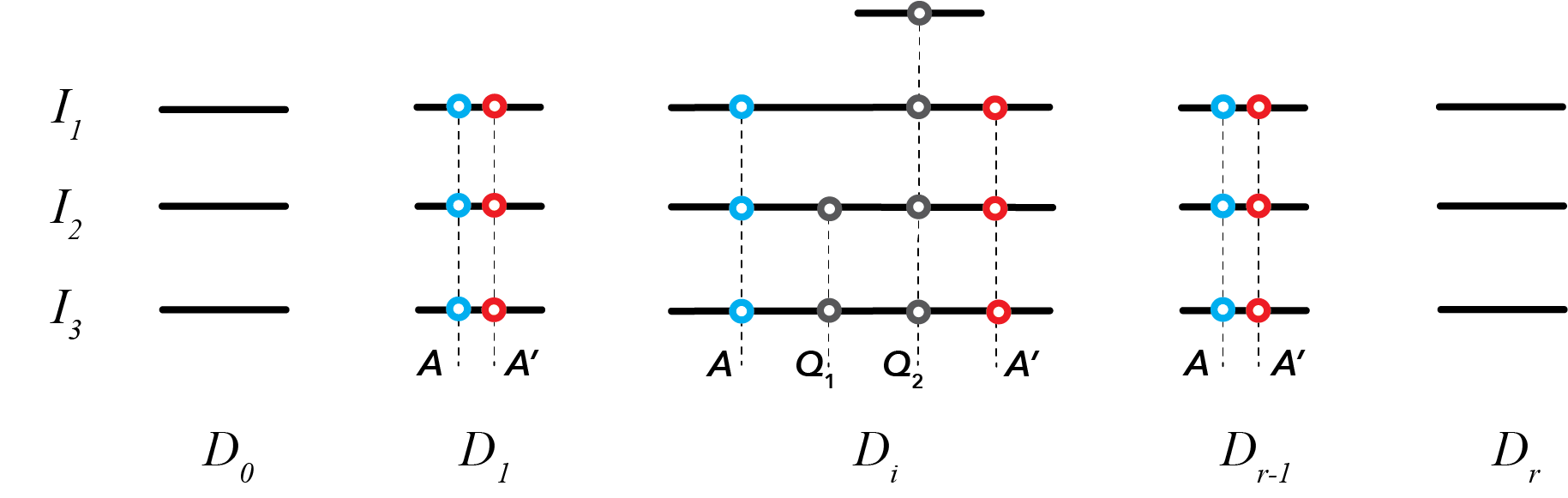}$$
\caption{The sequence of diagrams $D_i$ when  $\{A,A'\}=\{B,B'\}$.}\label{seq}
\end{figure}
 Note that each $Q_k$ either has all of its points on some of the $I_j$, or contains a point of $I_j$ for each $j$. Construct the  Gauss diagram ${D'_i}$ out of $D_i$ by
\begin{itemize}
\item
reversing the oriented cyclic order on the subset of each $Q_k$ that consists of the points of $I_1\cup\ldots\cup I_n$; 
\item
erasing the singular sets with the labels $A$ and $A'$.
\end{itemize}
Then, in the sequence of Gauss diagrams
$$D_0 ={D'_1}, {D'_2}, \ldots, {D'_{r-1}}$$
every diagram ${D'_i}$ with $i>1$ is either equal to the preceding diagram, or is obtained from it by a move of type $\Psi$. Since $D'_1$ is realizable, all the diagrams in the sequence are realizable, and each ${D'_i}$ has the same number of singular sets as $D_0$. Finally, the diagram $D'_{r-1}$ either coincides with $D_r$  or with $D_r^*$. Namely, it coincides with $D_r$ if the move from  $D_{r-1}$ to $D_r$ shrinks the intervals $I_j$ connecting $A$ with $A'$ and with $D_r^*$ if this move shrinks the complements of the $I_j$.

If $\{A,A'\}\cap\{B,B'\}$ consists of one element, without loss of generality, we assume that this element is $A'=B'$; as before, write $|A|=|A'|=|B|=n$. For each $0< i < r$ there are $n$ disjoint segments $I_1,\ldots I_n$ on the skeleton of the diagram $D_i$ that connect the points of $A$ with the points of $B$; each point of $A'$ lies between a point of $A$ and a point of $B$ on such a segment. Let $P_1,\ldots P_{a_i}$  be the singular sets of $D_i$ which have some of their points on the segments $I_j$ between the points of $A$ and $A'$ and let  $Q_1,\ldots Q_{b_i}$  be the singular sets of $D_i$ which have some of their points between the points of $A'$ and $B$.

\begin{figure}[h]
$$\includegraphics[width=400pt]{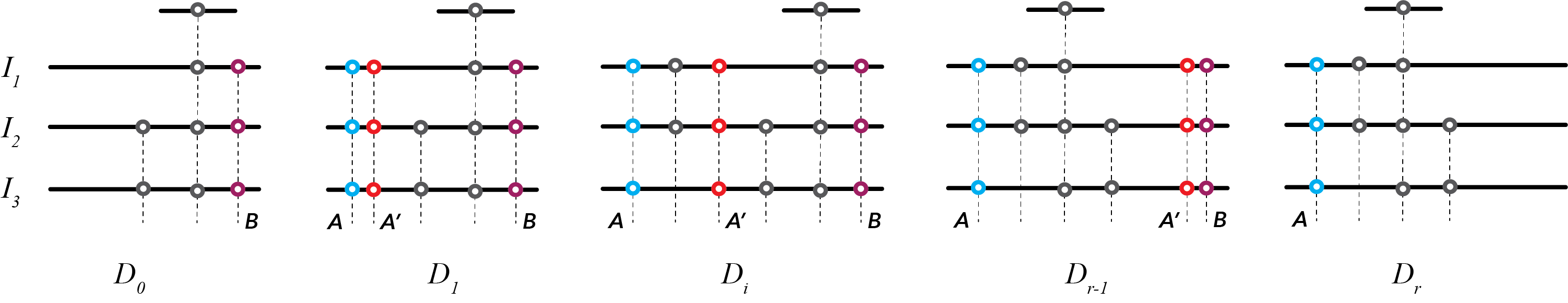}$$
\caption{The sequence of diagrams $D_i$ when $\{A,A'\}\cap\{B,B'\}$ has one element.}\label{seq2}.
\end{figure}

For each $0< i < r$ define the Gauss diagram $D'_i$ as follows:
\begin{itemize}
\item
in the diagram $D_i$, reverse the oriented cyclic order on the subset of each $P_k$ that consists of the points of $I_1\cup\ldots\cup I_n$;
\item
erase the singular sets with the labels $A$ and $A'$.
\end{itemize}
In a similar fashion, we define the diagram $D''_i$ 
\begin{itemize}
\item
by reversing in  $D_i$ the oriented cyclic order on the subset of each $P_k$ and each $Q_k$ that consists of the points of $I_1\cup\ldots\cup I_n$;
\item
erasing the singular sets with the labels $A$ and $B$.
\item
reversing the oriented cyclic order on $A'$, 
\end{itemize}
and the diagram $D'''_i$ by

\begin{itemize}
\item
reversing in  $D_i$ the oriented cyclic order on the subset of each $Q_k$ that consists of the points of $I_1\cup\ldots\cup I_n$;
\item
erasing the singular sets with the labels $A'$ and $B$.
\end{itemize}
Then, in the sequence of Gauss diagrams
$$D_0 ={D'_1}, {D'_2}, \ldots, {D'_{r-1}}, {D''_{r-1}}, , \ldots, {D''_2}, {D''_1}, {D'''_1}, {D'''_2} , \ldots, {D'''_{r-1}}=D_r, $$
each diagram is either equal to the preceding diagram, or is obtained from it by a move of type $\Psi$, or (as it is in the case of ${D''_{r-1}}$ and ${D'''_1}$) is obtained from it by renaming one of the singular sets. On the other hand, each ${D'_i}$, ${D''_i}$ and ${D'''_i}$ has the same number of singular sets as $D_0$.

\medskip
Finally, the case when $\{A,A'\}\cap\{B,B'\}=\emptyset$ is settled by a similar argument. 
Write $|A|=|A'|=m$ and $|B|=|B'|=n$, let $I_1,\ldots I_r$ be the disjoint intervals of the $D_i$ connecting the points of $A$ with those of $A'$ and let $J_1,\ldots J_s$ be the disjoint intervals connecting the points of $B$ with those of $B'$. (The intervals $I_k$ and $J_l$ may well not be disjoint).  Let $P_1,\ldots P_{a_i}$  be the singular sets of $D_i$ which have some of their points on the segments $I_k$ and let  $Q_1,\ldots Q_{b_i}$  be the singular sets of $D_i$ which have some of their points on the $J_l$. It may happen, of course, that $P_k=Q_l$ for some $k$ and $l$. 

For each $0< i < r$ define the Gauss diagram $D'_i$ as follows:
\begin{itemize}
\item
in the diagram $D_i$, reverse the oriented cyclic order on the subset of each $P_k$ that consists of the points of $I_1\cup\ldots\cup I_n$;
\item
erase the singular sets with the labels $A$ and $A'$.
\end{itemize}
In a similar fashion, we define the diagram $D''_i$ 
\begin{itemize}
\item
by reversing in  $D_i$ the oriented cyclic order on the subset of each $P_k$  that consists of the points of $I_1\cup\ldots\cup I_n$ and, afterwards,  on the subset of each $Q_k$  that consists of the points of $J_1\cup\ldots\cup J_m$.  
\item
erasing the singular sets with the labels $A, A', B$ and $B$.
\end{itemize}
and the diagram $D'''_i$ by
\begin{itemize}
\item
reversing in  $D_i$ the oriented cyclic order on the subset of each $Q_k$ that consists of the points of $J_1\cup\ldots\cup J_m$;
\item
erasing the singular sets with the labels $B$ and $B'$.
\end{itemize}
Then, in the sequence of Gauss diagrams
$$D_0 ={D'_1}, {D'_2}, \ldots, {D'_{r-1}}, {D''_{r-1}}, , \ldots, {D''_2}, {D''_1}, {D'''_1}, {D'''_2} , \ldots, {D'''_{r-1}}=D_r, $$
each diagram is either equal to the preceding diagram, or is obtained from it by a move of type $\Psi$, or, as it is in the case of ${D''_{r-1}}$ and ${D'''_1}$ is obtained from it by a move of type $\Phi$. The two moves of type $\Phi$ in this sequence are realizable by Lemma~\ref{rea}. Indeed, $D'_{r-1}$ is obtained from a realizable diagram $D'_1$ by moves of type $\Psi$ and ${D''_{r-1}}$ is then obtained by eliminating two adjacent singular sets by a move of type $\Phi.$  Similarly, we know that  ${D'''_{r-1}}$ is realizable, that ${D'''_1}$ is obtained from it by moves of type $\Psi$ and that ${D''_1}$ is obtained from ${D'''_1}$ by eliminating a pair of singular sets.

Finally, each ${D'_i}$ and ${D'''_i}$ has the same number of singular sets as $D_0$, while  ${D''_i}$ has two singular sets fewer than $D_0$.

\end{proof}

\end{document}